\documentclass[reqno]{amsart}

\usepackage{amsfonts,amssymb,latexsym,amsmath, amsxtra,graphicx}
\usepackage{verbatim,color}
\usepackage{mathtools}

\usepackage[OT2,T1]{fontenc}
\DeclareSymbolFont{cyrletters}{OT2}{wncyr}{m}{n}
\DeclareMathSymbol{\Sha}{\mathalpha}{cyrletters}{"58}

\theoremstyle{plain}
\newtheorem{theorem}{Theorem}[section]
\newtheorem*{theorem*}{Theorem}
\newtheorem{corollary}[theorem]{Corollary}
\newtheorem{conjecture}[theorem]{Conjecture}

\newtheorem{proposition}[theorem]{Proposition}
\newtheorem*{conjecture*}{Conjecture}
\newtheorem*{example*}{Example}

\theoremstyle{definition}

\theoremstyle{remark}
\newtheorem{remark}{Remark}

\newtheorem*{remark*}{Remark}
\newtheorem*{remarks*}{Remarks}

\numberwithin{equation}{section}

\makeatletter
\newcommand*{\rom}[1]{\expandafter\@slowromancap\romannumeral #1@}
\makeatother

\def\({\left(}
\def\){\right)}

\newcommand{\la}{\lambda}
\allowdisplaybreaks

\begin{document}

\title[Over-$(q,t)$-Binomial Coefficients]{An overpartition analogue of $q$-binomial coefficients, II: combinatorial proofs and $(q,t)$-log concavity}

\author{Jehanne Dousse}
\address{Institut f\"ur Mathematik, Universit\"at Z\"urich\\ Winterthurerstrasse 190, 8057 Z\"urich, Switzerland}
\email{jehanne.dousse@math.uzh.ch}

 \author{Byungchan Kim}
\address{School of Liberal Arts \\ Seoul National University of Science and Technology \\ 232 Gongneung-ro, Nowon-gu, Seoul,01811, Korea}
\email{bkim4@seoultech.ac.kr}


\date{\today}
\thanks{This research was supported by the International Research \& Development Program of the National Research Foundation of Korea (NRF) funded by the Ministry of Education, Science and Technology(MEST) of Korea (NRF-2014K1A3A1A21000358), the Forschungskredit of the University of Zurich, grant no. FK-16-098, and the STAR program number 32142ZM}

\begin{abstract}
 In a previous paper, we studied an overpartition analogue of Gaussian polynomials as the generating function for overpartitions fitting inside an $m \times n$ rectangle. Here, we add one more parameter counting the number of overlined parts, obtaining a two-parameter generalization $\overline{{m+n \brack n}}_{q,t}$ of Gaussian polynomials, which is also a $(q,t)$-analogue of Delannoy numbers. First we obtain finite versions of classical $q$-series identities such as the $q$-binomial theorem and the Lebesgue identity, as well as two-variable generalizations of classical identities involving Gaussian polynomials. Then, by constructing involutions, we obtain an identity involving a finite theta function and prove the $(q,t)$-log concavity of $\overline{{m+n \brack n}}_{q,t}$. We particularly emphasize the role of combinatorial proofs and the consequences of our results on Delannoy numbers. We conclude with some conjectures about the unimodality of $\overline{{m+n \brack n}}_{q,t}$.
\end{abstract}

\maketitle

\section{Introduction}
Gaussian polynomials (or $q$-binomial coefficients) are defined by 
\[
{m+n \brack n}_q = \frac{ (q)_{m+n} }{(q)_{m} (q)_{n}},
\]
where $(a)_{k} = (a;q)_{k} := \prod_{j=1}^{k} (1-aq^{j-1} )$ for $k \in \mathbb{N}_{0} \cup \{ \infty \}$. They are the generating functions for partitions fitting inside an $m \times n$ rectangle. In our previous paper \cite{DK}, we studied an overpartition analogue $\overline{{m+n \brack n}}_{q}$ of these polynomials as the generating function for the number of overpartitions fitting inside an $m \times n$ rectangle. We recall that an overpartition is a partition in which the last occurrence of each distinct number may be overlined \cite{LC}, the eight overpartitions of $3$ being
\[
3, \overline{3}, 2+1, \overline{2}+1, 2 + \overline{1}, \overline{2}+ \overline{1}, 1+1+1, 1+1+\overline{1}.
\]

In this paper, we add a variable $t$ counting the number of overlined parts in our over $q$-binomial coefficients and define
 \[
\overline{{m+n \brack n}}_{q,t} := \sum_{k,j \geq 0} \overline{p}(m,n,k,N) t^k q^N,
\]
where  $\overline{p}(m,n,k,N)$ counts the number of overpartitions of $N$, with $k$ overlined parts, fitting inside an $m \times n$ rectangle, i.e. with largest part $\leq m$ and number of parts $\leq n$. We call these two-variable polynomials $\overline{{m+n \brack n}}_{q,t}$ over-$(q,t)$-binomial coefficients or $(q,t)$-overGaussian polynomials. If we set $t=0$, meaning that no part is overlined, we obtain the classical $q$-binomial coefficients, and if we set $t=1$ we obtain the over $q$-binomial coefficients of \cite{DK}.  As we shall see in Section \ref{sec:basic}, the polynomials $\overline{{m+n \brack n}}_{q,t}$ are also $(q,t)$-analogues of the Delannoy numbers $D(m,n)$ \cite{Delannoy}. 

Again, by conjugation of the Ferrers diagrams, it is clear that
\[
\overline{{m+n \brack n}}_{q,t} = \overline{{m+n \brack m}}_{q,t}.
\]

Most of our results of \cite{DK} easily generalize to this new setting. Moreover, the new variable $t$ also allows us to do more precise combinatorial reasoning. Therefore in this paper we mainly focus on combinatorial proofs, which turn out to be very powerful and often simpler than $q$-theoretic proofs.

The limiting behavior of over-$(q,t)$-binomial coefficients is interesting, with
\[
\lim_{n \to \infty} \overline{{n \brack j}}_{q,t}  = \frac{ (-t q)_{j}}{(q)_{j}},
\]
as when $n$ tends to infinity, the restriction on the size of the largest part (or equivalently the number of parts) disappears. From this limiting behavior, we expect natural finite versions of identities in which overpartitions naturally arise. In this direction, we consider finite versions of  classical $q$-series identities. For example, we prove a finite version of the $q$-binomial theorem. 
\begin{theorem} \label{fin_qbinom_thm}
For every positive integer $n$, 
\begin{equation}\label{fin_qbinom_id}
\sum_{k \geq 0} \overline{{ n+k-1 \brack k}}_{q,t} z^k q^k =  \frac{(-tzq^2)_{n-1}}{(zq)_{n}}.
\end{equation}
\end{theorem}
By taking the limit as $n \to \infty$, we find that
\[
\sum_{k \geq 0} \frac{ (-tq)_{k}}{(q)_{k}} z^k q^k = \frac{ (-tzq^2)_{\infty}}{(zq)_{\infty}}.
\]
Replacing $z$ by $z/q$ and $t$ by $-t/q$ gives  the $q$-binomial theorem.

We also prove a finite version of a special case of the Rogers-Fine identity. 
 \begin{theorem} \label{fin_RF_thm}
For a positive integer $n$, 
\[
\sum_{k \geq 0} \overline{{ n+k-1 \brack k}}_{q,t} z^k q^k =  \sum_{k \geq 0} \frac{z^k q^{k^2+k} (-t zq^2)_{k}}{(zq)_{k+1}}  \( \overline{{ n-1 \brack k}}_{q,t} +t z q^{2k+2} \overline{{ n-2 \brack k}}_{q,t} \).
\]
\end{theorem}
By taking the limit as $n \to \infty$, we obtain
\[
\sum_{k \geq 0} \frac{(-tq)_{k}}{(q)_{k}} z^k q^k =  \sum_{k \geq 0} \frac{z^k q^{k^2+k} (-tzq^2)_{k} (-tq)_{k} }{ (q)_{k} (zq)_{k+1}}  \( 1 +t z q^{2k+2}   \),
\]
which is the case $a=1$ of the Rogers-Fine identity~\cite{Fine}
\[
\sum_{k \geq 0} \frac{ (-tq)_{k}}{ (aq)_{k}} z^k q^k = \sum_{k \geq 0} \frac{a^k z^k q^{k^2+k} (-tzq^2 /a)_{k} (-tq)_{k} }{ (aq)_{k} (zq)_{k+1}}  \( 1 +t z q^{2k+2}   \).
\]

We also prove the following very curious identity, which contains a truncated theta function.
\begin{theorem}\label{fin_theta}
For each nonnegative integer $n$,
\[
\sum_{k=0}^{n} (-1)^k \overline{{n \brack k}}_{q,1} =
\begin{cases}
0 &\text{if $n$ is odd,} \\
\sum_{j=-n/2}^{n/2} (-1)^j q^{j^2} &\text{if $n$ is even.}
\end{cases}
\]
\end{theorem}
This identity is interesting in several aspects. First of all, it is not clear at all how the cancellation occurs. Its proof is reminiscent of Franklin's proof on Euler's pentagonal number theorem~\cite[Theorem 1.6]{Abook}. Secondly, it also resembles Zagier's ``strange'' identity \cite{Za}:
\[
\sum_{k \geq 0} (q)_{k} = -\frac{1}{2} \sum_{k \geq 1} \chi (k) q^{(k^2 -1 )/24}.
\]
In our case, by taking the limit as $n$ goes to infinity, we obtain the ``formal'' identity 
\[
\sum_{k \geq 0} (-1)^k \frac{(-q)_{k}}{(q)_{k}} = \sum_{k \in \mathbb{Z}} (-1)^k q^{k^2}.
\]
Here by ``formal'' identity, we mean that the left-hand side does not converge as a power series in $q$. Thirdly, in a $q$-theoretic sense, Theorem \ref{fin_theta} is equivalent to
\[
\sum_{|j| \leq n} (-1)^j q^{j^2} = 2\sum_{j=0}^{n-1}\sum_{k=0}^{j} (-1)^j q^{k(k+1)/2} {2n - k \brack j}_q{j \brack k}_q + (-1)^n \sum_{k=0}^{n} q^{k(k+1)/2} {2n - k \brack n}_q {n \brack k}_q.
\]
Lastly, the involution to prove Theorem \ref{fin_theta} implies the following identity.

\begin{corollary} \label{Fine_Cor}
For each positive integer $n$, we have
 \[
1 + \sum_{k=1}^{n} (-q)^k \( \overline{{ n \brack k }}_{q,1} + \overline{{ n-1
\brack k-1 }}_{q,1} \) = \sum_{ |j| \le  \lfloor ( n+1 )/2 \rfloor }
(-1)^j q^{j^2}.
\]
\end{corollary}
Corollary \ref{Fine_Cor} is a finite version of a special case of Alladi's weighted partition theorem \cite{Alladi1}.

 We also study $q$-log concavity properties. In \cite{Butler}, Butler showed that $q$-binomial coefficients are $q$-log concave, namely that for all $0<k<n$,
\[
{n \brack k}_q^2 - {n \brack k-1}_q{n \brack k+1}_q
\]
has non-negative coefficients as a polynomial in $q$. Actually, Butler \cite[Theorem 4.2]{Butler} proved a much stronger result, namely that
\begin{equation}
\label{eq:butler}
 { n \brack k}_q  { n \brack \ell }_q -{ n \brack k-1}_q { n \brack \ell +1}_q
\end{equation}
has non-negative coefficients as a polynomial in $q$ for $0<k \leq \ell < n$. Here we prove that  over-$(q,t)$-binomial coefficients satisfy a generalization of this property, and therefore are also $(q,t)$-log concave.

\begin{theorem}
\label{qlog_thm}
For all $0<k \leq \ell < n$, 
\[
\overline{ { n \brack k}}_{q,t}  \overline{ { n \brack \ell }}_{q,t} - \overline{ { n \brack k-1}}_{q,t} \overline{ { n \brack \ell +1}}_{q,t}
\]
has non-negative coefficients as a polynomial in $t$ and $q$.
\end{theorem}

 Our proof is again combinatorial, as  we construct an injection to show the non-negativity.  The $q$-log concavity of $q$-binomial coefficients and of Sagan's $q$-Delannoy numbers \cite{Sagan}, as well as the log-concavity (and therefore unimodality) of Delannoy numbers follow immediately from the proof of Theorem \ref{qlog_thm}, as we shall see in Section~\ref{sec:qlogconcave}.

The remainder of this paper is organized as follows. In Section~\ref{sec:basic}, we study basic properties of over $q$-binomial coefficients and give connections with Delannoy numbers. Then in Section~\ref{sec:qseries}, we study finite versions of the $q$-binomial theorem, a special case of the Rogers-Fine identity and the Lebesgue identity. In section \ref{sec:gene}, we focus on two-variable generalizations of classical identities involving binomial coefficients. Then in Section \ref{sec:invo}, we give the involution proof of Theorem~\ref{fin_theta}. In Section \ref{sec:qlogconcave}, we prove Theorem \ref{qlog_thm} by constructing an involution and study its implications. In Section~\ref{sec:unimodal}, we conclude with some observations and conjectures concerning the unimodality of the over-$(q,t)$-binomial coefficients $\overline{{m+n \brack n}}_{q,t}$.

\section{Basic properties and connection to Delannoy numbers}
\label{sec:basic}

The Delannoy numbers \cite{Delannoy} $D(m,n)$, also sometimes called Tribonacci numbers \cite{Alladi3}, are the number of paths from $(0,0)$ to $(m,n)$ on a rectangular grid, using only East, North and North-East steps, namely steps from $(i,j)$ to $(i+1,j)$, $(i,j+1)$, or $(i+1,j+1)$. Let $\mathcal{D}_{m,n}$ be the set of such paths. For a path $p \in \mathcal{D}_{m,n}$, we define the weight of each of its steps $p_k$ as
\[
wt(p_k) := 
\begin{cases}
0, &\text{ if it goes from $(i,j)$ to $(i+1,j)$,} \\
i, &\text{ if it goes from $(i,j)$ to $(i,j+1)$,} \\
i+1, &\text{ if it goes from $(i,j)$ to $(i+1,j+1)$.}
\end{cases}
\]
Then we define the weight $wt(p)$ of $p$ to be the sum of the weights of its steps, and $d(p)$ to be the number of North-East steps in $p$. By mapping North-East steps to overlined parts, we obtain a bijection between Ferrers diagrams of overpartitions fitting inside a $m \times n$ rectangle and Delannoy paths from the origin to $(m,n)$. Therefore, we can see that over-$(q,t)$-binomial coefficient are generating functions for Delannoy paths.

\begin{proposition}
For non-negative integers $m$ and $n$,
\[
\overline{{ m+n \brack n }}_{q,t} = \sum_{ p \in \mathcal{D}_{m,n}} t^{d(p)} q^{wt(p)}.
\]
\end{proposition}

In this sense, we can say that over-$(q,t)$-binomial coefficients are $(q,t)$-analogues of Delannoy numbers, which generalize the $q$-Delannoy numbers introduced by Sagan \cite{Sagan} (after exchanging $t$ and $q$),
\[
D_q (m,n) = \sum_{ p \in \mathcal{D}_{m,n}} q^{d(p)}.
\]
In particular when $q=t=1$ we have
$$\overline{{ m+n \brack n }}_{1,1} = D(m,n).$$
A different $q$-analogue of Delannoy numbers has been given by Ramirez in \cite{Ramirez}.

Most of our results of \cite{DK} generalize to the new setting with the additional variable $t$. It is sufficient to keep track of the number of overlined parts in the original proofs. Here we present two of them which have an interesting connection with Delannoy numbers.
 Now we give an exact formula for $\overline{{m+n \brack n}}_{q,t}$.
\begin{theorem}
\label{th:gene} 
For non-negative integers $m$ and $n$, 
\begin{equation}
\label{eq:formula}
\overline{{m+n \brack n}}_{q,t} = \sum_{k=0}^{\min\{m,n\}} t^k q^{\frac{k(k+1)}{2}} \frac{(q)_{m+n-k}}{(q)_k(q)_{m-k}(q)_{n-k}}.
\end{equation}
\end{theorem}
\begin{proof}
As in~\cite{DK}, let $\overline{G}(m,n,k)$ denote the generating function for overpartitions fitting inside an $m \times n$ rectangle and having exactly $k$ overlined parts.
We have
\[
\overline{G}(m,n,k) =  q^{\frac{k(k+1)}{2}} {m \brack k}_q {n+m-k \brack n-k}_q
= q^{\frac{k(k+1)}{2}} \frac{(q)_{m+n-k}}{(q)_k(q)_{m-k}(q)_{n-k}}.
\]
Since $\overline{G}(m,n,k)$ is non-zero if and only if $0 \leq k \leq \min\{m,n\}$, we have
\[
\overline{{m+n \brack n}}_{q,t} =\sum_{k=0}^{\min\{m,n\}} t^k \overline{G}(m,n,k) = \sum_{k=0}^{\min\{m,n\}} t^k q^{\frac{k(k+1)}{2}} \frac{(q)_{m+n-k}}{(q)_k(q)_{m-k}(q)_{n-k}}.
\]
\end{proof}

The case $t=0$ gives the classical formula for Gaussian polynomials and the case $t=1$ corresponds to Theorem 1.1 in \cite{DK}. Lemma 3 in \cite{AllBer} is essentially another formulation of Theorem \ref{th:gene}, but their proof is more complicated as it involves several $q$-series identities, while ours is purely combinatorial. Moreover, when $t=q=1$, we obtain the following classical formula for Delannoy numbers:
$$D(m,n)= \sum_{k=0}^{\min\{m,n\}} {n \choose k} {m+n-k \choose n}.$$
Note that using $q$-multinomial coefficients
$${a+b+c \brack a,b,c}_q := \frac{(q)_{a+b+c}}{(q)_a(q)_b(q)_c},$$
we can rewrite \eqref{eq:formula} as
\begin{equation}
\label{eq:trinom}
\overline{{m+n \brack n}}_{q,t} = \sum_{k=0}^{\min\{m,n\}} t^k q^{\frac{k(k+1)}{2}} {m+n-k \brack k,m-k,n-k}_q.
\end{equation}

In the same way, the analogues of Pascal's triangle of \cite{DK} can also be generalized.
\begin{theorem} \label{PropPascal}
For positive integers $m$ and $n$, we have
\begin{align}
\label{pa1}
\overline{{m+n \brack n}}_{q,t} &= \overline{{m+n -1\brack n-1}}_{q,t} +q^n \overline{{m+n-1 \brack n}}_{q,t} + tq^n \overline{ {m+n-2 \brack n-1}}_{q,t},\\
\label{pa2}
\overline{{m+n \brack n}}_{q,t} &= \overline{{m+n -1\brack n}}_{q,t} +q^m \overline{{m+n-1 \brack n-1}}_{q,t} + tq^m \overline{{m+n-2 \brack n-1}}_{q,t}.
\end{align}
\end{theorem}
Again $t=0$ gives the classical recurrences for $q$-binomial coefficients, $t=1$ gives Theorem 1.2 of~\cite{DK}, and $t=q=1$ gives the classical recurrence for Delannoy numbers:
$$D(m,n)=D(m-1,n)+D(m,n-1)+D(m-1,n-1).$$

We also obtain $(q,t)$-analogues of two other classical formulas for Delannoy numbers.
Recall that the basic hypergeometric series $_{r}\phi_{s}$ are defined by 
\[
_{r}\phi_{s} (a_1,a_2,\ldots,a_r;b_1,\ldots,b_s ;q,z) := \sum_{n \geq 0} \frac{ (a_1)_{n} (a_2)_{n} \cdots (a_r)_{n} }{ (q)_{n} (b_1)_{n} \cdots (b_s)_{n}} \left[ (-1)^n q^{n(n-1)/2} \right]^{1+s-r} z^n.
\]

We can express over-$(q,t)$-binomial coefficients using a basic hypergeometric series.
\begin{proposition}
For all $m,n$ positive integers,
\[
\overline{ {m+n \brack n }}_{q,t} = {{ m+n \brack n }_{q} }    {_{2}\phi_{1}} ( q^{-n}, q^{-m} ; q^{-n-m} ; q, -tq).
\]
\end{proposition}
\begin{proof}
We may assume $m \geq n$, as otherwise we could consider the conjugate of Ferrers diagram of the overpartitions. Using the fact that
\[
(q^{-n} ; q)_{k} = \frac{ (q;q)_{n}}{(q;q)_{n-k}} (-1)^k q^{\binom{k}{2} - nk},
\]
we derive that
\begin{align*}
{_{2}\phi_{1}} ( q^{-n}, q^{-m} ; q^{-n-m} ; q, -tq) &= \sum_{k =0}^{n} \frac{ (q^{-n})_{k} (q^{-m})_{k} }{ (q)_{k} (q^{-n-m})_{k}} (-tq)^k \\
&= \frac{ (q)_{n} (q)_{m}}{(q)_{n+m}} \sum_{k=0}^{n} \frac{ (q)_{m+n-k} t^k q^{k(k+1)/2}}{(q)_{n-k} (q)_{m-k} (q)_{k}} \\
&=  \frac{ (q)_{n} (q)_{m}}{(q)_{n+m}} \overline{ {m+n \brack n }}_{q,t}
\end{align*}
as desired.
\end{proof}

By setting $t=q=1$, we can recover the well-known formula for Delannoy numbers
\[
D(m,n) = \binom{m+n}{n} {_{2}F_{1}} (-n,-m; -m-n ; -1),
\]
where $_{2}F_{1}$ is a hypergeometric function. 

Moreover, from \cite[Appendix III.8]{GR} we have a transformation formula for the terminating series 
\[
_{2}\phi_{1} (q^{-n}, b ; c ; q,z) = \frac{ (c/b)_{n} }{(c)_{n}} b^n {_{3}\phi_{1}} (q^{-n}, b, q/z ; bq^{1-n}/c ; q , z/c ).
\]
By setting $b=q^{-m}$, $c=q^{-n-m}$, and $z=-tq$, we find another expression for over-$(q,t)$-binomial coefficients.

\begin{proposition} \label{3rd_over_prop}
For all non-negative integers $m$ and $n$, 
\[
\overline{ {m+n \brack n} }_{q,t} = \sum_{k=0}^{\min\{m ,n \}} x^k q^{k(k+1)/2} (-1/x)_{k}
{m \brack k}_{q} {n \brack k}_{q}.
\]
\end{proposition} 
By setting $q=t=1$, we can obtain another well-known formula for Delannoy numbers
\[
D(m,n) = \sum_{k=0}^{\min\{m,n\}} 2^k \binom{m}{k} \binom{n}{k}.
\]
Actually, the bijection given in \cite[Theorem 1.1]{LC} gives a combinatorial proof for Proposition \ref{3rd_over_prop}. As details are lengthy and we do not use this bijection later, we omit details here.

\section{Finite versions of classical $q$-series identities}
\label{sec:qseries}

\subsection{The $q$-binomial theorem}
 In this section we use the $(q,t)$-overGaussian polynomials $\overline{{m+n \brack n}}_{q,t}$ to prove new finite versions of classical $q$-series identities. Recall the $q$-binomial theorem.
\begin{theorem}[$q$-binomial theorem]
For $|t|,|z|<1$, 
\[
\sum_{k \geq 0} \frac{(t)_{k}}{(q)_{k}} z^k = \frac{(tz)_{\infty}}{(z)_{\infty}}.
\]
\end{theorem}
We start by giving two different finite versions of the $q$-binomial theorem involving over-$(q,t)$-binomial coefficients.
We first prove combinatorially Theorem \ref{fin_qbinom_thm}.

\begin{proof}[Proof of Theorem \ref{fin_qbinom_thm}]
 Notice that $z^k q^k$ generates a column of $k$ unoverlined $1$'s. We append the partition generated by $ \overline{{ n+k-1 \brack k}}_{q,t}$ to the right of these $1$'s. Therefore, we find that the left-hand side of \eqref{fin_qbinom_id} is the generating function for the number of overpartitions with largest part $\leq n$ and no overlined $1$, where the exponent of $z$ counts the number of parts and the exponent of $t$ counts the number of overlined parts. It is clear that the right-hand side of \eqref{fin_qbinom_id} generates the same partitions.
\end{proof}

Moreover Proposition 3.1 of \cite{DK} can be easily generalized by keeping track of the number of overlined parts in the original proof, and gives another finite version of the $q$-binomial theorem.
\begin{theorem}[Generalization of Proposition 3.1 of \cite{DK}] \label{fin_qbi}
For every positive integer $n$, we have
\[
\frac{(-tzq)_{n}}{(zq)_{n}} = 1+ \sum_{k \geq 1} z^k q^k \left(  \overline{\left[ \begin{matrix} n+k-1 \\ k \end{matrix} \right]}_{q,t} + t \overline{ \left[ \begin{matrix} n+k-2 \\ k-1 \end{matrix} \right]}_{q,t} \right).
\]
\end{theorem}

By letting $n$ tend to infinity, we obtain the following.

\begin{corollary}[Generalization of Corollary 3.2 of \cite{DK}]
\label{cor}
Let $\overline{p} (n,k,\ell)$ be the number of overpartitions of $n$ with $k$ parts and $\ell$ overlined parts. Then,
\[
\sum_{n,k,\ell \geq 0} \overline{p} (n,k,\ell) z^k t^{\ell} q^n =\frac{(-tzq)_{\infty}}{(zq)_{\infty}} = 1+ \sum_{k \geq 1} \frac{ z^k q^{k} (-t)_{k}}{(q)_{k}}.
\]
\end{corollary}

Now replacing $z$ by $z/q$ and $t$ by $-t$ in the above gives the $q$-binomial theorem.

\subsection{A special case of the Rogers-Fine identity}
We now turn to the proof of Theorem \ref{fin_RF_thm}, which uses Durfee decomposition.
\begin{proof}[Proof of Theorem \ref{fin_RF_thm}]
We first observe that for every positive integer $n$, 
\begin{align*}
\sum_{k \geq 0} \overline{{ n+k-1 \brack k}}_{q,t} z^k q^k = &\sum_{k \geq 0} \frac{z^k q^{k^2+k} (-tq^2)_{k}}{(zq)_{k+1}} \overline{{ n-1 \brack k}}_{q,t}\\
+ &\sum_{k \geq 0} \frac{t z^k q^{k^2+k} (-tq^2)_{k-1}}{(zq)_{k}} \overline{{ n-2 \brack k-1}}_{q,t}.
\end{align*}
The left-hand side is the generating function for the number of overpartitions with largest part $\leq n$ and no overlined $1$, where the exponent of $z$ counts the number of parts and the exponent of $t$ counts the number of overlined parts, as in the proof of Theorem \ref{fin_qbinom_thm}. Now we consider the Durfee rectangle of size $(k+1) \times k$. We can distinguish two cases according to whether the bottom-right corner of Durfee rectangle is overlined or not.  When it is not overlined, $z^k q^{k^2 + k}$ generates the Durfee rectangle. Moreover, $\frac{ (-t zq^2)_{k}}{ (zq)_{k+1} }$ generates the overpartition below the Durfee rectangle and $\overline{{n-1 \brack k}}_{q,t}$ generates the overpartition to the right of the Durfee rectangle. When the bottom-right corner of the Durfee rectangle is overlined, $t z^k q^{k^2+k}$ generates the Durfee rectangle. Since the parts below the Durfee rectangle are less then $k+1$ in this case, they are generated by $\frac{ (-t  zq^2)_{k-1}}{ (zq)_{k} }$. Moreover, there could be no further overlined $k+1$, so $\overline{{ n- 2 \brack k-1}}_{q,t}$ generates the overpartition to the right of the Durfee rectangle. By replacing $k$ by $k+1$ in the second sum, we obtain the desired identity.
\end{proof}

\subsection{The Lebesgue identity}
Finally we also have a generalization of Sylvester's identity \cite{Syl}, which is a finite version of the Lebesgue identity. We define
\[
S(n;t,y,q) := 1 + \sum_{j \geq 1} \left( t \overline{\left[ \begin{matrix} n-1 \\ j-1 \end{matrix} \right] }_{q,t} \frac{(-tyq)_{j-1}}{(yq)_{j-1}} y^j q^{j^2} +  \overline{\left[ \begin{matrix} n \\ j \end{matrix} \right]}_{q,t} \frac{(-tyq)_{j}}{(yq)_{j}} y^j q^{j^2} \right).
\]
\begin{theorem}[Generalization of Theorem 3.4 of \cite{DK}] \label{fin_Leb}
For any positive integer $n$, 
\begin{equation} \label{Sylid}
S(n;t,y;q)= \frac{(-tyq)_{n}}{(yq)_{n}}.
\end{equation}
\end{theorem}

The new variable $t$ allows us to deduce Lebesgue's identity from Theorem \ref{fin_Leb}.
\begin{corollary}[Lebesgue's identity]
\label{th:leb}
For $|q| <1$, 
\[
\sum_{k \geq 0} \frac{ (-tq)_{k} q^{k(k+1)/2}}{(q)_{k}} = \frac{ (-tq^2 ;q^2)_{\infty}}{ (q;q^2)_{\infty}}.
\]
\end{corollary}
\begin{proof}
 In \eqref{Sylid}, we replace $q$ by $q^2$, $t$ by $tq$, and $y$ by $1/q$. Then, by taking the limit as $n$ goes to infinity, we find that
\begin{align*}
\frac{(-tq^2 ; q^2)_{\infty}}{ (q;q^2)_{\infty}} 
&= 1+ \sum_{j \geq 1} tq \frac{ (-tq^3;q^2)_{j-1} (-tq^2 ; q^2)_{j-1}}{(q^2;q^2)_{j-1} (q;q^2)_{j-1}} q^{2j^2-j} + \frac{ (-tq^3;q^2)_{j} (-tq^2 ;q^2)_{j}}{ (q^2;q^2)_{j} (q;q^2)_{j}} q^{2j^2 - j} \\
&= 1 + \sum_{j \geq 1} tq \frac{(-xq^2)_{2j-2}}{(q)_{2j-2}} q^{2j^2 - j} + \frac{(-tq^2)_{2j}}{(q)_{2j}} q^{2j^2 - j} \\
&=1+ \sum_{j \geq 1} \frac{ (-tq)_{2j-1}}{(q)_{2j-1}} \frac{tq(1-q^{2j-1})}{1+tq} q^{2j^2 - j}  + \frac{(-tq)_{2j}}{(q)_{2j}} \frac{1+tq^{2j+1}}{1+tq} q^{2j^2-j} \\
&=1+ \sum_{j \geq 1} \frac{ (-tq)_{2j-1}}{(q)_{2j-1}} \( 1 - \frac{1+tq^{2j}}{1+tq} \)q^{2j^2 - j} + \frac{(-tq)_{2j}}{(q)_{2j}} \( q^{2j} - \frac{1-q^{2j}}{1+tq} \) q^{2j^2 - j } \\
&= 1+ \sum_{j \geq 1} \frac{ (-tq)_{2j-1}}{(q)_{2j-1}} q^{2j^2 - j} + \frac{(-tq)_{2j}}{(q)_{2j}} q^{2j^2 + j} \\
&= \sum_{j \geq 0} \frac{ (-tq)_{j} q^{j(j+1)/2}}{(q)_{j}}. 
\end{align*}
\end{proof}
Thus we can see Theorem \ref{fin_Leb} as a finite version of the Lebesgue identity. Two different finite versions were given by Rowell \cite{Rowell}, and Alladi and Berkovich \cite{AllBer}, respectively. Alladi \cite{Alladi2} gave another proof of the Lebesgue identity in terms of partitions into distinct odd parts.

\section{Generalizations of $q$-binomial coefficients identities}
\label{sec:gene}
In this section, we prove  two-variable generalizations of Gaussian polynomial identities. As a first example, by tracking the number of parts, one can easily see that the following identity \cite[Eqn. (3.3.9)]{Abook} holds:
\begin{equation}
\label{eq:and}
\sum_{j=0}^{n} q^j { m+j \brack j}_q = { n+m+1 \brack m+1}_q,
\end{equation}
which is a $q$-analogue of the classical identity
\[
\sum_{j=0}^{n} \binom{m+j}{j} = \binom{n+m+1}{m+1}.
\]
By tracking the number of overlined and non-overlined parts separately, we can prove the following two-parameter generalization of \eqref{eq:and}.
\begin{proposition}
For positive integers $m$ and $n$,
\[
\overline{{m+n+1 \brack m+1}}_{q,t} =1 + \sum_{j=1}^{n} q^{j} \( \overline{{m+j \brack j}}_{q,t} + t \overline{{m+j-1 \brack j-1}}_{q,t} \).
\]
\end{proposition}
By taking the limit when $m \to \infty$, we also find that
\[
\frac{(-tq)_{n}}{(q)_{n}} = 1 + \sum_{j=1}^{n} q^j\( \frac{ (-tq)_{j}}{(q)_{j}} + t \frac{ (-tq)_{j-1}}{(q)_{j-1}} \) = 1+ \sum_{j=1}^{N} \frac{ (-t)_{j} q^j}{(q)_{j}}.
\]

By setting $q=t=1$, we find that  
\[
D(m+1,n) = 1 + \sum_{j=1}^{n} \( D(m,j) + D(m,j-1) \).
\]

Secondly, we find an over-Gaussian polynomial generalization of the identity \cite[Eqn. 3.3.10]{Abook}  
\[
\sum_{k=0}^{h} { n \brack k}_q { m \brack h-k}_q q^{(n-k)(h-k)} = {m+n \brack h}_q,
\]
which is a $q$-analogue of the classical identity
\[
\sum_{k=0}^{h} \binom{n}{k}\binom{m}{h-k} = \binom{n+m}{h}.
\]
\begin{proposition} \label{2nd_over_prop}
For positive integers $m, n \geq  h$, 
\[
\sum_{k=0}^{h}  q^{(n-k)(h-k)} \( \overline{{ n \brack k}}_{q,t} \overline{{ m \brack h-k}}_{q,t} + t \overline{{ n-1 \brack k}}_{q,t} \overline{{ m-1 \brack h-k-1}}_{q,t} \)  = \overline{ {m+n \brack h}}_{q,t}.
\]
\end{proposition}

\begin{proof}
For an overpartition $\la$ generated by the right hand side, we consider the largest rectangle of the form $(n-k) \times (h-k)$ fitting inside the Ferrers diagram of $\la$, i.e. its Durfee rectangle of size $(n-k) \times (h-k)$. It is clear that such a $k$ is uniquely determined, and as $\la$ has at most $h$ parts, $k$ is between $0$ and $h$. We have two cases according to whether the bottom right corner of the Durfee rectangle is overlined or not. In the case where it is non-overlined, the overpartition on the right side of the Durfee rectangle does fit inside a $(m-h+k) \times (h-k)$ rectangle and the overpartition below the Durfee rectangle is inside a $(n-k) \times k$ rectangle. The generating function of such partitions is $q^{(n-k)(h-k)}  \overline{{ n \brack k}}_{q,t} \overline{{ m \brack h-k}}_{q,t} .$  In the case where the bottom right corner is overlined, we can see that the overpartition on the right side should be inside a $(m-h+k) \times (h-k-1)$ rectangle and the overpartition below the Durfee rectangle fits inside a $(n-k-1) \times k$ rectangle, the generating function of such partitions equals $t  q^{(n-k)(h-k)}\overline{{ n-1 \brack k}}_{q,t} \overline{{ m-1 \brack h-k-1}}_{q,t}$.
\end{proof}

By setting $q=t=1$ and $m=m+h$ in Proposition \ref{2nd_over_prop}, we find that for $n,m \geq h  >0$,
\[
D(m+n,h) = \sum_{k=0}^{h} \( D(n-k,k)D(m+k,h-k) + D(n-k-1,k) D(m+k,h-k-1) \).
\]

Finally, in \cite{PS},  Prellberg and Stanton used the following identity
\[
\frac{1}{(x)_{n}} = \sum_{m=0}^{n-1} \( { n+m -1 \brack 2m } q^{2m^2} \frac{x^{2m}}{(x)_{m}} + { n+m \brack 2m+1 } q^{2m^2+m} \frac{x^{2m+1}}{(x)_{m+1}} \)
\]
to prove that for all $n$, the coefficients of
 \[
 (1-q) \frac{1}{(q^n )_{n}} + q
 \]
are non-negative.
  
By employing Durfee rectangle dissection according to whether the size of the Durfee rectangle is $(m+1) \times 2m$ or $m \times (2m-1)$ and whether the corner of Durfee rectangles is overlined or not, we can deduce an overpartition version.

\begin{theorem}
For any positive integer $n$,
\[
\begin{aligned}
\frac{ (-tzq)_{n} }{ (zq)_{n} } &= 1+ \sum_{m=1}^{n-1} \( \overline{ {n+m-1 \brack 2m } }_{q,t} + t \overline{ {n+m-2 \brack 2m-1 } }_{q,t}  \) z^{2m} q^{2m^2+2m} \frac{ (-tzq)_{m} }{(zq)_{m}}   \\
&+\sum_{m=1}^{n} \overline{ {n+m-1 \brack 2m-1 } }_{q,t} z^{2m-1} q^{2m^2 - m} \frac{ (-tzq)_{m} }{(zq)_{m}} 
\\& + \sum_{m=1}^{n} \overline{ {n+m-2 \brack 2m-2 } }_{q,t} t z^{2m-1} q^{2m^2-m} \frac{ (-tzq)_{m-1} }{(zq)_{m-1}} .
\end{aligned}
\]
\end{theorem}

By taking the limit $n \to \infty$, we find that
\[
\begin{aligned}
\frac{ (-tzq)_{\infty} }{ (zq)_{\infty} } &= 1+ \sum_{m=1}^{\infty} \( \frac{ (-tq)_{2m}}{(q)_{2m}} + t \frac{ (-tq)_{2m-1}}{(q)_{2m-1}} \) z^{2m} q^{2m^2+2m} \frac{ (-tzq)_{m} }{(zq)_{m}}   \\
&\quad +\sum_{m=1}^{\infty} \frac{ (-tq)_{2m-1}}{(q)_{2m-1}} z^{2m-1} q^{2m^2 - m} \frac{ (-tzq)_{m} }{(zq)_{m}} 
\\&\quad + \sum_{m=1}^{\infty} \frac{ (-tq)_{2m-2}}{(q)_{2m-2}} t z^{2m-1} q^{2m^2-m} \frac{ (-tzq)_{m-1} }{(zq)_{m-1}} \\
&=\sum_{m=0}^{\infty}  z^{2m} q^{2m^2+2m} \frac{(-tq)_{2m} (-tzq)_{m} }{(q)_{2m} (zq)_{m}} \(1+tzq^{m+1} \) \\
&\quad+ \sum_{m=1}^{\infty} z^{2m-1} q^{2m^2 - m} \frac{ (-tq)_{2m-1}(-tzq)_{m} }{(q)_{2m-1} (zq)_{m}} \( 1 + tzq^{3m} \). 
\end{aligned}
\]
This can be viewed as an overpartition analogue of 
\[
(-zq)_{\infty} = \sum_{k \geq 0} z^k q^{k(3k+1)/2} (1 + zq^{2k+1} ) \frac{(-zq)_{k}}{(q)_{k}},
\]
which becomes Euler's pentagonal number theorem when $z=-1$.

Numerics suggest an overpartition analogue of the result of Prellberg and Stanton.

\begin{conjecture}
\label{conj:prellberg}
For all positive integers $n$, the coefficients of
 \[
 (1-q) \frac{ (-q^n)_{n}}{(q^n)_{n}} + q
 \]
are non-negative. 
\end{conjecture}

\section{The involution proof of Theorem \ref{fin_theta}}
\label{sec:invo}
We now prove Theorem \ref{fin_theta}, using an involution similar to Franklin's proof of Euler's Pentagonal Numbers Theorem.

For convenience, we allow non-overlined $0$ as a part. Then, we can interpret the coefficient of $q^N$ in
\[
(-1)^k  \overline{{n \brack k}}_{q,1}
\]
as the number of overpartitions of $N$ into ``exactly'' $k$ parts $\leq n-k$ with weight $(-1)^k$. Let $\mathcal{O}_{k,n}$ be the set of above-mentioned overpartitions. For an overpartition $\lambda \in \mathcal{O}_{k,n}$ with $k \leq n$, we denote by $\pi$ the overpartition below its Durfee square and by $\mu$ the conjugate of the overpartition on the right of the Durfee square. If the size of the Durfee square is $d$ ($\le k$), then $\pi$ has $k-d$ parts and $\mu$ has less than $N-k-d$ parts. Define $s(\pi)$ to be the smallest nonzero part of $\pi$ and $s(\mu)$ to be the smallest part of $\mu$. (Note that $\mu$ does not have 0 as a part). If there is no nonzero part in $\pi$ or $\mu$ then we define $s(\pi) = 0$ or $ s(\mu) =0$ accordingly. We also define $s_2 (\pi)$ (resp. $s_2(\mu)$) to be the second smallest nonzero part of $\pi$ (resp. $\mu$).

We build a sign-reversing involution $\phi$ on $\mathcal{O}_{n} = \cup_{0 \le k \le n} \mathcal{O}_{k,n}$ as follows:
\begin{itemize}
\item[Case 1.] If $s(\pi) = s(\mu) =0$, then $\phi(\lambda)=\lambda$. This case is invariant under this map.
\item[Case 2.] If $s(\pi)=0$ and $s(\mu)>0$ or $s(\pi) > s(\mu)$, $\phi(\lambda)$ is obtained by moving $s(\mu)$ below $s(\pi)$. Then, the resulting partition is in $\mathcal{O}_{k+1,n}$ since it has now $k+1$ parts and the size of the largest part is decreased by $1$, and thus it does not violate the maximum part condition for $\mathcal{O}_{k+1,n}$.  

\begin{figure}
\includegraphics[scale=1]{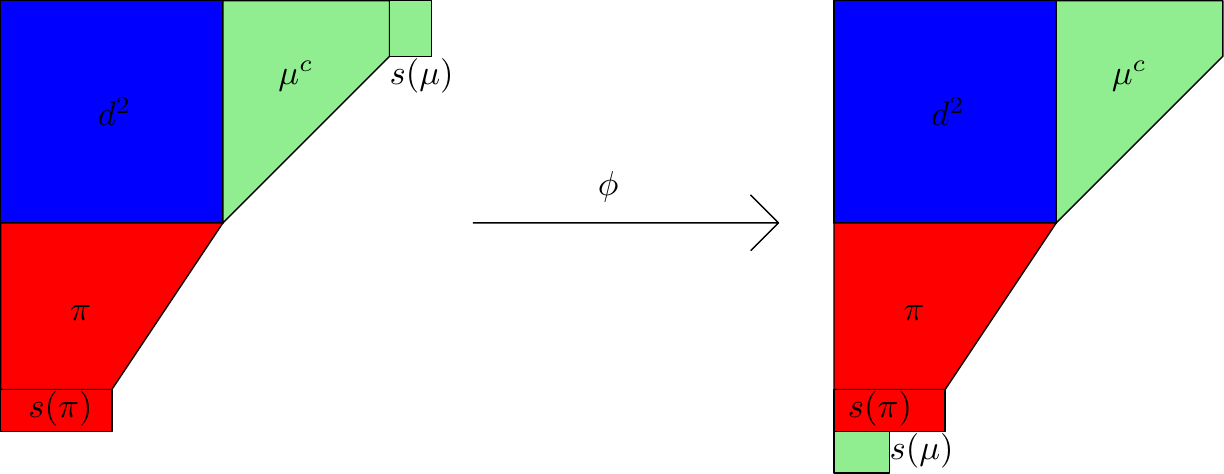}
\caption{The case 2 of the involution $\phi$}
\end{figure}

\item[Case 3.] If $s(\pi) < s(\mu)$, 
\begin{itemize}
\item[Case 3.1] if $s_2(\pi) = s(\pi)$ and $s(\pi)$ is not overlined, we overline $s(\pi)$ and $s_2(\pi)$ and move $s(\pi)$ to the right of $s(\mu)$.

\item[Case 3.2] if $s_2(\pi) > s(\pi)$ or $s(\pi)$ is overlined, $\phi(\lambda)$ is obtained by moving $s(\pi)$ to the right of $s(\mu)$.

\begin{figure}
\includegraphics[scale=1]{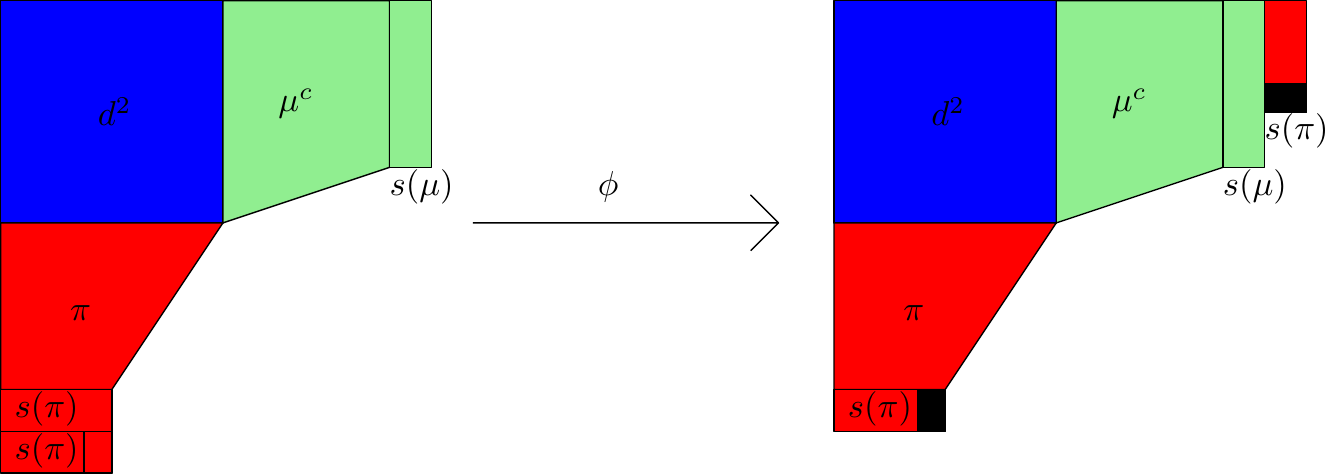}
\caption{The case 3.1 of the involution $\phi$}
\end{figure}
\end{itemize}
In both cases, the resulting overpartition is in $\mathcal{O}_{k-1,N}$.

\item[Case 4.] $s(\pi) = s(\mu)$. We have to consider different subcases according to whether $s(\pi)$ and $s(\mu)$ are overlined or not.

For convenience, we define $\chi ( a ) = 1$ if $a$ is an overlined part and $\chi(a)=0$ if $a$ is a non-overlined part.
\begin{itemize}
\item[Case 4.1.] If $\chi( s(\mu) )=\chi(s(\pi))=1$, we move $s(\mu)$ below $s(\pi)$ and un-overline both $s(\pi)$ and $s(\mu)$. Note that the resulting overpartition is in $\mathcal{O}_{k+1,n}$.
\item[Case 4.2.] If $\chi( s(\mu)) =1$ and $\chi(s(\pi))=0$, we move $s(\mu)$ below $s(\pi)$. The resulting overpartition is in $\mathcal{O}_{k+1,n}$.
\item[Case 4.3.] If $\chi ( s(\mu) ) = 0$ and $\chi(s(\pi)) =1 $, we move $s (\pi)$ to the right of $s(\mu)$ and this gives an overpartition in $\mathcal{O}_{k-1,n}$.
\item[Case 4.4.] If $\chi ( s(\mu) ) = \chi(s(\pi))=0$
\begin{itemize}
\item[Case 4.4.1.] if $s_2(\pi) = s(\pi)$, then overline $s_2(\pi)$ and $s(\pi)$ and move $s (\pi)$ to the right of $s(\mu)$.
\item[Case 4.4.2.] if $s_2(\pi) > s(\pi)$ or $s_2 (\pi) =0$, we move $s (\pi)$ to the right of $s(\mu)$.
\end{itemize}
In both cases, this gives an overpartition in $\mathcal{O}_{k-1,n}$. 
\end{itemize}
\end{itemize}

Before proving $\phi$ is an involution, here we give one example.

\begin{example*}
For an overpartition $(5,\overline{5},3,2,0) \in \mathcal{O}_{5,10}$, the size of the Durfee square is $3$, $\pi = (2,0)$, and $\mu = (2, \overline{2})$. Thus, $s(\pi)=s(\mu)=2$. Since $\chi(\mu)=1$, we move $\overline{2}$ in $\mu$ below $s(\pi)$. As a result, we have a new overpartition with $\pi = (2, \overline{2},0)$ and $\mu = (2)$, which gives the overpartition $(4,4,3,2,\overline{2},0) \in \mathcal{O}_{6,10}$. Note that $\phi ((4,4,3,2,\overline{2},0)) = (5,\overline{5},3,2,0) \in \mathcal{O}_{5,10}$ as we expected. 
\end{example*}

Now we prove that this is true in general.

\begin{proposition}
The map $\phi$ is an involution.
\end{proposition}
\begin{proof}
We need to prove that for every overpartition $\lambda$ in $\mathcal{O}_{n}$, we have $\phi(\phi(\lambda))=\lambda.$ Here also, we need to distinguish several cases.
\begin{itemize}
\item[Case 1.] If $s(\pi) = s(\mu) =0$, then $\phi(\lambda)=\lambda$, so $\phi(\phi(\lambda))=\lambda.$
\item[Case 2.] If $s(\pi) > s(\mu)$,
	\begin{itemize}
	\item[Case 2.a.] if $s_2(\mu) > s(\mu)$, then $\phi(\lambda)$ is obtained by moving $s(\mu)$ below $s(\pi)$. Thus $\phi(\lambda)$ is in the case 3.2 and we obtain $\phi(\phi(\lambda))$ by moving $s(\mu)$ back to its initial place. Therefore $\phi(\phi(\lambda))=\lambda.$
	\item[Case 2.b.] if $s_2(\mu) = s(\mu)$ and $s(\mu)$ is overlined, then $\phi(\lambda)$ is in the case 4.3 and $\phi(\phi(\lambda))=\lambda.$
	\item[Case 2.c.] if $s_2(\mu) = s(\mu)$ and $s(\mu)$ is non-overlined, then $\phi(\lambda)$ is in the case 4.4.2 and $\phi(\phi(\lambda))=\lambda.$
	\end{itemize}
\item[Case 3.] If $s(\pi) < s(\mu)$,
	\begin{itemize}
	\item[Case 3.a.] if $s_2(\pi) = s(\pi)$ and $s(\pi)$ is not overlined, $\lambda$ is in the case 3.1 and $\phi(\lambda)$ is obtained by overlining $s(\pi)$ and $s_2(\pi)$ and moving $s(\pi)$ to the right of $s(\mu)$. Thus $\phi(\lambda)$ is in the case 4.1 and we obtain $\phi(\phi(\lambda))$ by moving $s(\pi)$ back to its initial place and un-overlining $s(\pi)$ and $s_2(\pi)$ again. Therefore $\phi(\phi(\lambda))=\lambda.$
	\item[Case 3.b.] if $s_2(\pi) = s(\pi)$ and $s(\pi)$ is overlined, $\lambda$ is in the case 3.2 and $\phi(\lambda)$ is obtained by moving $s(\pi)$ to the right of $s(\mu)$. Thus $\phi(\lambda)$ is in the case 4.2 and we obtain $\phi(\phi(\lambda))$ by moving $s(\pi)$ back to its initial place. Therefore $\phi(\phi(\lambda))=\lambda.$
	\item[Case 3.c.] if $s_2(\pi) > s(\pi)$, $\lambda$ is in the case 3.2, $\phi(\lambda)$ is in the case 2, and $\phi(\phi(\lambda))=\lambda.$
	\end{itemize}
\item[Case 4.] If $s(\pi) = s(\mu)$,
	\begin{itemize}
	\item[Case 4.1.] if $\chi( s(\mu) )=\chi(s(\pi))=1$,
		\begin{itemize}
		\item[Case 4.1.a.] if $s_2(\mu) = s(\mu)$, then $\phi(\lambda)$ is obtained by moving $s(\mu)$ under $s(\pi)$ and un-overlining both. Thus $\phi(\lambda)$ is in case 4.4.1 and we get $\phi(\phi(\lambda))$ by moving $s(\mu)$ back to its initial place and overlining $s(\pi)$ and $s(\mu)$ again. Therefore $\phi(\phi(\lambda))=\lambda.$
		\item[Case 4.1.b.] if $s_2(\mu) > s(\mu)$, then $\phi(\lambda)$ is in case 3.1 and $\phi(\phi(\lambda))=\lambda.$
		\end{itemize}
	\item[Case 4.2.] if $\chi( s(\mu)) =1$ and $\chi(s(\pi))=0$, 
		\begin{itemize}
		\item[Case 4.2.a.] if $s_2(\mu) = s(\mu)$, then $\phi(\lambda)$ is obtained by moving $s(\mu)$ under $s(\pi)$. Thus $\phi(\lambda)$ is in case 4.3 and $\phi(\phi(\lambda))=\lambda.$
		\item[Case 4.2.b.] if $s_2(\mu) > s(\mu)$, then $\phi(\lambda)$ is in case 3.2 and $\phi(\phi(\lambda))=\lambda.$
		\end{itemize}
	\item[Case 4.3.] if $\chi ( s(\mu) ) = 0$ and $\chi(s(\pi)) =1 $,
		\begin{itemize}
		\item[Case 4.3.a.] if $s_2(\pi) = s(\pi)$, then $\phi(\lambda)$ is obtained by moving $s(\pi)$ to the right of $s(\mu)$. Thus $\phi(\lambda)$ is in case 4.2 and $\phi(\phi(\lambda))=\lambda.$
		\item[Case 4.2.b.] if $s_2(\pi) > s(\pi)$, then $\phi(\lambda)$ is in case 2 and $\phi(\phi(\lambda))=\lambda.$
		\end{itemize}
	\item[Case 4.4.] if $\chi ( s(\mu) ) = \chi(s(\pi))=0$
		\begin{itemize}
		\item[Case 4.4.1.] if $s_2(\pi) = s(\pi)$, then $\phi(\lambda)$ is obtained by overlining $s_2(\pi)$ and $s(\pi)$ and moving $s (\pi)$ to the right of $s(\mu)$. Thus $\phi(\lambda)$ is in case 4.1 and we get $\phi(\phi(\lambda))$ by moving $s(\pi)$ back to its initial place and un-overlining $s(\pi)$ and $s(\mu)$ again. Therefore $\phi(\phi(\lambda))=\lambda.$
		\item[Case 4.4.2.] if $s_2(\pi) > s(\pi)$, then $\phi(\lambda)$ is in case 2 and $\phi(\phi(\lambda))=\lambda.$
		\end{itemize}
	\end{itemize}
\end{itemize}
Thus in every case, $\phi(\phi(\lambda))=\lambda.$
\end{proof}

Now we are finally ready to prove Theorem \ref{fin_theta}.
\begin{proof}
From the sign reversing involution $\phi$, we see that only square overpartitions survive after pairing $\la \in \mathcal{O}_n$ and $\phi(\la) \in \mathcal{O}_n$. Moreover, the square overpartition of $j^2$ (with $0 \leq j \leq \lfloor n/2 \rfloor$) is in $\mathcal{O}_{k,n}$ for $k$ from $j$ to $n-j$. Thus, the sum of weights is
\[
\sum_{ k = j}^{n-j} (-1)^k =
\begin{cases}
0, &\text{if $n$ is odd,} \\
(-1)^j, &\text{if $n$ is even,}
\end{cases}
\]
as the summation runs over $n-2j+1$ consecutive integers from $j$. By considering overlined and non-overlined square overpartitions, we arrive at
\[
\sum_{k=0}^{n} (-1)^k  \overline{{n \brack k}}_{q,1} =
\begin{cases}
0, &\text{if $n$ is odd,} \\
1 + 2 \sum_{j=1}^{n/2} (-1)^j q^{j^2}, &\text{if $n$ is even,}
\end{cases}
\]
as there is no overlined partition for the empty partition.
\end{proof}

A proof of Corollary \ref{Fine_Cor} follows from the simple observation that the right-hand side of Corollary \ref{Fine_Cor} corresponds to having exactly $k$ positive parts in the involution instead of $k$ non-negative parts.
 
Finally, by setting $q=t=1$ in Theorem \ref{fin_theta}, we obtain the following.
\begin{corollary}
For all positive integers $n$,
\[
\sum_{k=0}^{n} (-1)^k  D(n-k,k) = 
\begin{cases}
0, &\text{if $n$ is odd,} \\
-1, &\text{if $n \equiv 2 \pmod{4}$,} \\
1, &\text{if $n \equiv 0 \pmod{4}$.} 
\end{cases}
\]
\end{corollary}

\section{$(q,t)$-log-concavity of the over $q$-binomial coefficients}
\label{sec:qlogconcave}

In this section, we prove Theorem \ref{qlog_thm} by constructing an involution. Before starting the proof, we introduce some notation. Let $\overline{\mathcal{P}} $ denote the set of overpartitions of non-negative integers, and $\overline{\mathcal{P}} ( m,n )$ the set of overpartitions fitting inside a $m \times n$ rectangle. We also write $\#_o (\la)$ for the number of overlined parts in $\la$ and $|\la|$ for the weight of $\la$  (i.e. the sum of its parts).

\begin{proof}[Proof of Theorem \ref{qlog_thm}]
To prove Theorem \ref{qlog_thm}, we want to find an injection $\phi$ from $\overline{\mathcal{P}} (n-k+1, k-1 ) \times \overline{\mathcal{P}} ( n-\ell-1, \ell +1  )$ to 
$\overline{\mathcal{P}} (n-k, k ) \times \overline{\mathcal{P}} ( n-\ell, \ell  )$, such that, if $\phi(\la,\mu)=(\eta,\rho),$ then $|\la|+|\mu| = |\eta|+|\rho|$ and $\#_o (\la)+ \#_o (\mu) = \#_o (\eta) + \#_o (\rho).$

We generalize the proof in \cite{Butler} to overpartitions. We define two maps $\mathcal{A}$ and $\mathcal{L}$ on $\overline{\mathcal{P}} \times \overline{\mathcal{P}}$, and take $\phi$ to be the restriction of $\mathcal{L} \circ \mathcal{A}$ to the domain  $\overline{\mathcal{P}} (n-k+1, k-1 ) \times \overline{\mathcal{P}} ( n-\ell-1, \ell +1 ).$
We obtain the injectivity of $\phi$ by showing that
\begin{enumerate}
\item $\mathcal{A}$ and $\mathcal{L}$ are involutions on $\overline{\mathcal{P}} \times \overline{\mathcal{P}}$,
\item $ \mathcal{A} \left( \overline{\mathcal{P}} (n-k+1, k-1) \times \overline{\mathcal{P}} ( n-\ell-1, \ell +1  ) \right) \subset \overline{\mathcal{P}} (n-k, k-1) \times \overline{\mathcal{P}} (n-\ell, \ell +1 ),$
\item  $ \mathcal{L} \left( \overline{\mathcal{P}} (n-k, k-1) \times \overline{\mathcal{P}} ( n-\ell, \ell +1) \right) \subset \overline{\mathcal{P}} (n-k, k) \times \overline{\mathcal{P}} (n-\ell, \ell ).$
\end{enumerate}

Let us start by defining $\mathcal{A}$. For a given overpartition pair $(\la, \mu) \in \overline{\mathcal{P}} \times \overline{\mathcal{P}}$, we define $I$ by the largest integer satisfying
\begin{equation} \label{t_cond}
\la_I  -\mu_{I+1} \geq
\begin{cases}
 \ell -k +1, &\text{ if $\la_I$ is not overlined,} \\
 \ell -k +2, &\text{ if $\la_I$ is overlined,} 
\end{cases}
\end{equation}
where we define $\mu_{i+1} =0$ if $\la_i >0$, but the number of parts in $\mu$ is less than $i+1$. If there is no such $I$, we define $I=0$. Now we define
\[
\mathcal{A} (\la, \mu) =(\gamma, \tau),
\]
where 
\[
\begin{aligned}
\gamma &:= ( \mu_1 + (\ell -k +1), \ldots, \mu_{I} + (\ell -k +1), \la_{I+1}, \la_{I+2}, \ldots    ), \\
\tau &:= (\la_1 - (\ell -k +1), \ldots, \la_{I} - (\ell -k +1), \mu_{I+1}, \mu_{I+2}, \ldots ).
\end{aligned}
\]
Note that if $\lambda_i$ (resp. $\mu_i$), $i \leq I$ was overlined (resp. non-overlined) in $\lambda$ (resp. $\mu$), then $\la_i - (\ell -k +1)$ (resp. $\mu_i + (\ell -k +1)$) is overlined (resp. non-overlined) in $\tau$ (resp. $\gamma$).

Before defining $\mathcal{L}$, let us introduce two maps $\mathcal{S}$ and $\mathcal{C}$ on $\overline{\mathcal{P}} \times \overline{\mathcal{P}}$ by
\[
\mathcal{S} (\la, \mu) := (\mu, \la) \quad\text{and}\quad \mathcal{C} (\la,\mu) := (\la^c, \mu^c).
\]
Then we define $\mathcal{L}$ as
\[
\mathcal{L} := \mathcal{S} \circ \mathcal{C} \circ \mathcal{A} \circ \mathcal{C} \circ \mathcal{S}.
\]

We now want to verify that $(i)$ is satisfied. Since $\mathcal{S}$ and $\mathcal{C}$ are involutions on $\overline{\mathcal{P}} \times \overline{\mathcal{P}}$, we only need to show that $\mathcal{A}$ is an involution.

First of all, let us verify that $\mathcal{A}$ is well defined, i.e. if  $\mathcal{A} (\la, \mu) =(\gamma, \tau)$ and $(\la, \mu) \in \overline{\mathcal{P}} \times \overline{\mathcal{P}}$, then $\gamma$ and $\tau$ are also overpartitions.
By definition $( \mu_1 + (\ell -k +1), \ldots, \mu_{I} + (\ell -k +1)), (\la_{I+1}, \la_{I+2}, \ldots ),$ $(\la_1 - (\ell -k +1), \ldots, \la_{I} - (\ell -k +1))$ and $(\mu_{I+1}, \mu_{t+2}, \ldots )$ are overpartitions so we only need to check that
\begin{equation}
\label{isover1}
\mu_{I} + (\ell -k +1) \geq \begin{cases} \la_{I+1} \text{ if $\mu_{I} + (\ell -k +1)$ is not overlined}\\
\la_{I+1} +1 \text{ if $\mu_{I} + (\ell -k +1)$ is overlined,}
\end{cases}
\end{equation}
and
\begin{equation}
\label{isover2}
\la_{I} - (\ell -k +1) \geq \begin{cases} \mu_{I+1} \text{ if $\la_{I} - (\ell -k +1)$ is not overlined}\\
\mu_{I+1} +1 \text{ if $\la_{I} - (\ell -k +1)$ is overlined.}
\end{cases}
\end{equation}
Equation \eqref{isover2} is clear by \eqref{t_cond}. Let us turn to \eqref{isover1}.
By definition of $I$, we have
\begin{equation*}
\mu_{I+2} + (\ell -k +1) \geq
\begin{cases}
 \lambda_{I+1} +1 , &\text{ if $\lambda_{I+1}$ is not overlined,} \\
\lambda_{I+1}, &\text{ if $\lambda_{I+1}$ is overlined,} 
\end{cases}
\end{equation*}
If $\mu_I$ is not overlined, then $\mu_I \geq 
\mu_{I+2}$, so $$\mu_{I} + (\ell -k +1) \geq \la_{I+1}.$$
If $\mu_I$ is overlined, then by definition of an overpartition $\mu_I \geq 
\mu_{I+2} +1$, so $$\mu_{I} + (\ell -k +1) \geq \la_{I+1}+1.$$
This completes the verification of \eqref{isover1}.

Now we want to check that $\mathcal{A}$ is an involution. Let $(\la, \mu) \in \overline{\mathcal{P}} \times \overline{\mathcal{P}}$ and $(\gamma,\tau) =  \mathcal{A} (\la,\mu).$ We want to show that $\mathcal{A} (\gamma,\tau)= (\la,\mu).$
By definition of $I$ and $\mathcal{A}$, the parts with indices $\geq I+1$ of $\gamma$ (resp. $\tau$) are exactly the same as those of $\la$ (resp. $\mu$) and will therefore not be moved when we apply $\mathcal{A}$ again. Therefore the only thing left to check is that
$$\gamma_I  -\tau_{I+1} \geq
\begin{cases}
 \ell -k +1, &\text{ if $\gamma_I$ is not overlined,} \\
 \ell -k +2, &\text{ if $\gamma_I$ is overlined,} 
\end{cases}$$
that is that
$$\mu_I +(\ell -k+1) -\mu_{I+1} \geq
\begin{cases}
 \ell -k +1, &\text{ if $\mu_I$ is not overlined,} \\
 \ell -k +2, &\text{ if $\mu_I$ is overlined,} 
\end{cases}$$
which is clear by definition of an overpartition.
Thus the $I$ of $(\gamma,\tau)$ is the same as the one of $(\lambda,\mu)$, and $\mathcal{A}(\gamma,\tau) = (\lambda,\mu).$
The point $(i)$ is proved.

Then, point $(ii)$ is obvious from the definition of $\mathcal{A}$.

Finally let us verify $(iii)$. We have, by definition of $\mathcal{S}, \mathcal{C}$ and $\mathcal{A},$
\begin{align*}
\mathcal{S} \left( \overline{\mathcal{P}} (n-k, k-1) \times \overline{\mathcal{P}} ( n-\ell, \ell +1) \right) &= \overline{\mathcal{P}} ( n-\ell, \ell +1) \times \overline{\mathcal{P}} (n-k, k-1),\\
\mathcal{C} \left( \overline{\mathcal{P}} ( n-\ell, \ell +1) \times \overline{\mathcal{P}} (n-k, k-1) \right) &= \overline{\mathcal{P}} (\ell +1, n-\ell) \times \overline{\mathcal{P}} (k-1, n-k),\\
\mathcal{A} \left( \overline{\mathcal{P}} (\ell +1, n-\ell) \times \overline{\mathcal{P}} (k-1, n-k) \right) &\subset \overline{\mathcal{P}} (\ell, n-\ell) \times \overline{\mathcal{P}} (k, n-k),\\
\mathcal{C} \left( \overline{\mathcal{P}} (\ell, n-\ell) \times \overline{\mathcal{P}} (k, n-k) \right) &= \overline{\mathcal{P}} (n-\ell, \ell) \times \overline{\mathcal{P}} (n-k,k),\\
\mathcal{S} \left( \overline{\mathcal{P}} (n-\ell, \ell) \times \overline{\mathcal{P}} (n-k,k) \right) &= \overline{\mathcal{P}} (n-k,k) \times \overline{\mathcal{P}} (n-\ell, \ell).
\end{align*}
Thus $(iii)$ is satisfied.

\end{proof}

Here we give an example to illustrate the map $\phi = \mathcal{L} \circ \mathcal{A} =  \mathcal{S} \circ \mathcal{C} \circ \mathcal{A} \circ \mathcal{C} \circ \mathcal{S} \circ \mathcal{A}$.
\begin{example*}
When $n=10$, $k=4$, and $\ell=5$, we consider the partition pair $(\la, \mu) \in \overline{\mathcal{P}} (7,3) \times \overline{\mathcal{P}} (4,6)$, where
\[
\la = (\overline{7},6,4), \;\; \text{and}\;\; \mu=(4,\overline{4},3,3,2,2).
\] 
Then, we see that
\[
\begin{aligned}
(\la, \mu) 
& \xmapsto{\; \mathcal{A} \; } ( (6, \overline{6}, 4 ), (\overline{5},4,3,3,2,2) ) \\
& \xmapsto{\; \mathcal{S} \;} (  (\overline{5},4,3,3,2,2) ,  (6, \overline{6}, 4 ) ) \\
& \xmapsto{\; \mathcal{C} \; } ( ( 6, 6, 4, 2, \overline{1} ), (3,3,3, 3,2,\overline{2} ) ) \\
& \xmapsto{\; \mathcal{A}\; } ( (5,5, 4,2,\overline{1}), (4,4, 3, 3,2,\overline{2} ) ) \\
& \xmapsto{\; \mathcal{C} \; } ( (\overline{5}, 4,3,3, 2 ), (6, \overline{6}, 4,2 ) ) \\
& \xmapsto{\; \mathcal{S} \;} (   (6, \overline{6}, 4,2 ),  (\overline{5}, 4,3,3, 2 ) ),
\end{aligned}
\]
which is in $\overline{\mathcal{P}} (6,4) \times \overline{\mathcal{P}} (5, 5)$ as desired. 
\end{example*}

If we forbid overlined parts (i.e. if we set $t=0$), the above proof becomes Butler's proof of \eqref{eq:butler}.

From Theorem \ref{qlog_thm}, we can deduce several interesting corollaries.

\begin{corollary}
\label{cor:qlog}
The over-$(q,t)$-binomial coefficients are $(q,t)$-log-concave, namely for all $0<k < n$,
\[
\overline{ { n \brack k}}_{q,t}^2- \overline{ { n \brack k-1}}_{q,t} \overline{ { n \brack k+1}}_{q,t}
\]
has non-negative coefficients as a polynomial in $q$ and $t$.
\end{corollary}
By setting $t=0$, we obtain Butler's result on the $q$-log-concavity of  $q$-binomial coefficients.

Recall that we have shown that $\mathcal{L}$ is an injection  from $\overline{\mathcal{P}} (n-k, k-1) \times \overline{\mathcal{P}} ( n-\ell, \ell +1)$ to $\overline{\mathcal{P}} (n-k, k) \times \overline{\mathcal{P}} (n-\ell, \ell )$. From this we obtain the following.

\begin{theorem}
\label{cor2:qlog}
For all $0<k \leq \ell < n$, 
\[
\overline{ { n \brack k}}_{q,t} \overline{ { n \brack \ell }}_{q,t} - \overline{ { n-1 \brack k-1}}_{q,t}  \overline{ { n+1 \brack \ell+1 }}_{q,t} 
\]
has non-negative coefficients as a polynomial in $q$ and $t$.
\end{theorem}

 By setting $q=t=1$ in Theorems \ref{qlog_thm} and \ref{cor2:qlog}, we deduce the following result on Delannoy numbers.

\begin{corollary}
\label{cor3:logdelannoy}
For all $0 <k \leq \ell <n$, we have
\[
\begin{aligned}
D(n-k,k) D(n-\ell,\ell) &\geq D(n-k+1,k-1) D(n-\ell-1,\ell+1), \\
D(n-k,k) D(n-\ell,\ell) &\geq D(n-k,k-1) D(n-\ell,\ell+1).
\end{aligned}
\]

\end{corollary}

Now by setting $\ell=k$ and $n=n+k$ in Corollary \ref{cor3:logdelannoy} this yields the log-concavity of Delannoy numbers, which also implies their unimodality.

\begin{corollary}
\label{cor:logdelannoy}
For all $n>k>0$, the Delannoy numbers $D(n,k)$ satisfy
\[
\begin{aligned}
D(n,k)^2 &\geq D(n+1,k-1) D(n-1,k+1), \\
D(n,k)^2 &\geq D(n,k-1) D(n,k+1).
\end{aligned}
\]
In particular, the Delannoy numbers $D(n,k)$ are log-concave.
\end{corollary}

Similarly, by setting $q=1$ and $t=q$, in Theorems \ref{qlog_thm} and \ref{cor2:qlog}, we deduce the following result on Sagan's $q$-Delannoy numbers.

\begin{corollary}
\label{cor2:logqdelannoy}
For all $0 <k \leq \ell <n$, we have
\[
\begin{aligned}
D_q(n-k,k) D_q(n-\ell,\ell) &\geq D_q(n-k+1,k-1) D_q(n-\ell-1,\ell+1), \\
D_q(n-k,k) D_q(n-\ell,\ell) &\geq D_q(n-k,k-1) D_q(n-\ell,\ell+1).
\end{aligned}
\]

\end{corollary}

And by setting $\ell=k$ and $n=n+k$ in Corollary \ref{cor2:logqdelannoy}, we obtain the $q$-log-concavity of Sagan's $q$-Delannoy numbers.

\begin{corollary}
\label{cor:logqdelannoy} 
For all $n>k>0$, Sagan's $q$-Delannoy numbers $D_q(n,k)$ satisfy that
\begin{align*}
&D_q (n,k)^2 - D_q(n+1,k-1) D_q(n-1,k+1) 
\intertext{and}
&D_q (n,k)^2 - D_q(n,k-1) D_q(n,k+1)
\end{align*}
have non-negative coefficients as polynomials in $q$. In particular Sagan's $q$-Delannoy numbers $D_q(n,k)$ are $q$-log-concave.
\end{corollary}

Moreover, we can also generalize Corollary 4.5 of \cite{Butler} to over-$(q,t)$-binomial coefficients.
\begin{corollary}
\label{cor:butler}
For $0 \leq k-r \leq k \leq \ell \leq \ell+r \leq n,$
\[
\overline{ { n \brack k}}_{q,t} \overline{ { n \brack \ell}}_{q,t} - \overline{ { n \brack k-r}}_{q,t} \overline{ { n \brack \ell+r}}_{q,t}
\]
has non-negative coefficients as a polynomial in $t$ and $q$.
\end{corollary}
\begin{proof}
The proof is similar to the one in \cite{Butler}.
By Theorem \ref{qlog_thm}, all the terms of the telescoping sum
\begin{align*}
\overline{ { n \brack k}}_{q,t}& \overline{ { n \brack \ell}}_{q,t} - \overline{ { n \brack k-r}}_{q,t} \overline{ { n \brack \ell+r}}_{q,t}\\&= \sum_{i=0}^{r-1} \left( \overline{ { n \brack k-i}}_{q,t} \overline{ { n \brack \ell+i}}_{q,t} - \overline{ { n \brack k-i-1}}_{q,t} \overline{ { n \brack \ell+i+1}}_{q,t} \right)
\end{align*}
have non-negative coefficients.
\end{proof}

As usual, setting $q=t=1$ yields some interesting result on Delannoy numbers.
\begin{corollary}
\label{cor:butlerdelannnoy}
For $0 \leq k-r \leq k \leq \ell \leq \ell+r \leq n,$
\[
D(n-k,k) D(n-\ell,\ell) \geq D(n-k+r,k-r) D(n-\ell+r,\ell+r).
\]
\end{corollary}

\section{Unimodality conjectures}
\label{sec:unimodal}
We now present a few conjectures and observations about the unimodality of over-$(q,t)$-binomial coefficients.
Recall that a polynomial $p(x)= a_0 + a_1 x + \cdots + a_r x^r$ is \emph{unimodal} if there is an integer $\ell$ (called the \emph{peak}) such that
\[
a_0 \leq a_1 \leq \cdots \leq a_{\ell-1} \leq a_{\ell} \geq a_{\ell+1} \geq \cdots \geq a_r.
\]
It is well-known that Gaussian polynomials \cite{Syl2} and $q$-multinomial coefficients \cite[Theorem 3.11]{Abook} are unimodal.

We extend this definition to polynomials in two variables. We say that a polynomial $P(q,t) = \sum_{k=0}^r \sum_{n=0}^s a_{k,n}t^kq^n$ is \emph{doubly unimodal} if
\begin{enumerate}
\item for every fixed $k \in \{0, \dots , r\}$, the coefficient of $t^k$ in $P(q,t)$ is unimodal in $q$, that is there exists an integer $\ell$ such that
$$a_{k,0} \leq a_{k,1} \leq \cdots \leq a_{k,\ell-1} \leq a_{k,\ell} \geq a_{k,\ell+1} \geq \cdots \geq a_{k,s},$$
\item for every fixed $n \in \{0, \dots , s\}$, the coefficient of $q^n$ in $P(q,t)$ is unimodal in $t$, that is there exists an integer $\ell'$ such that
$$a_{0,n} \leq a_{1,n} \leq \cdots \leq a_{\ell'-1,n} \leq a_{\ell',n} \geq a_{\ell'+1,n} \geq \cdots \geq a_{r,n}.$$
\end{enumerate}

Computer experiments suggest that the following conjectures are true.
\begin{conjecture}
\label{conj:doubleunimodal}
For every positive integers $m$ and $n$, the over-$(q,t)$-binomial coefficient $\overline{{ m +n \brack n }}_{q,t}$ is doubly unimodal.
\end{conjecture}
\begin{remark}
By the formula \eqref{eq:trinom} and using the fact that $q$-multinomial coefficients are unimodal, we can easily deduce that part $(i)$ of the definition is satisfied. Therefore the challenging part of the conjecture is to prove that for every $N$, the coefficient of $q^N$ in $\overline{{m \brack n}}_{q,t}$ is unimodal in $t$.
\end{remark}

\begin{conjecture}
\label{conj:unimodal}
For every positive integers $m$ and $n$, $\overline{{ m +n \brack n }}_{q,1}$ is unimodal in $q$.
\end{conjecture}
\begin{remark}
 Conjecture \ref{conj:doubleunimodal} doesn't immediately imply Conjecture \ref{conj:unimodal}, as the peaks in $q$ are not the same for each $t^k$. Therefore, even if they might be related, the two conjectures are of independent interest.
\end{remark}

We illustrate our conjectures for $m=n=4$ in Table 1.

\begin{table}[h]
\begin{center}
\begin{tabular}{c|c|c}
$n$ & The coefficient of $q^n$ & The coefficient of $q^n$ when $t=1$  \\ \hline
 0 & 1 & 1 \\
1 & $1+t$ & 2 \\
2 & $2+2t$ & 4 \\
3 & $3+4t+t^2$ & 8  \\
4 & $5+7t+2t^2$ & 14 \\
5 & $5+10t+5t^2$ & 20 \\
6 & $7+13t+7t^2 +t^3$ & 28 \\
7 & $7+16t+11t^2+2t^3$ & 36 \\
8 & $8+ 17t + 12 t^2 + 3t^3$ & 40 \\
9 & $7+17t+ 14t^2 + 4t^3$ & 42 \\
10 & $7+ 16t+12t^2 + 4t^3 + t^4$ & 40 \\
11 & $5 + 13t + 11t^2 + 3t^3$ & 32 \\
12 & $5 + 10t + 7t^2 + 2t^3$ & 24\\
13 & $3 + 7t + 5t^2 + t^3 $&  16 \\
14 & $2 + 4t + 2t^2$ &  8\\
15 & $1+ 2t + t^2$ & 4\\
16 & $1+ t$ & 2 
\end{tabular}
\end{center}
\label{999}
\caption{The coefficients of $\overline{{8 \brack 4 }}_{t,q}$.}
\end{table}

\begin{remark}
Pak and Panova \cite{Pak} recently proved that the classical $q$-binomial coefficients are strictly unimodal. Experiments show that it should also be the case for $\overline{{ m +n \brack N }}_{q,1}$, and for the coefficients of $q^N$ in $\overline{{m+n \brack n}}_{q,t}$ (as a polynomial in $t$). However it is not the case for the coefficients of $t^k$ in $\overline{{m+n \brack n}}_{q,t}$ (as a polynomial in $q$).
\end{remark}

\section*{Acknowledgements}
The authors thank Krishna Alladi, Manjul Bhargava, Alex Berkovich,  Bruce Berndt, and Ali Uncu for their valuable comments.


\begin{thebibliography}{30}

\bibitem{Alladi1}
K. Alladi, \emph{A new combinatorial study of the Rogers-Fine identity and a related partial theta series},  Int. J. Number Theory {\bf 5} (2009), 1311--1320. 

\bibitem{Alladi3}
K. Alladi, V. E. Hoggatt, \emph{On Tribonacci numbers and related functions}, Fibonacci Quart. {\bf 15} (1977), 42--45.

\bibitem{Alladi2}
K. Alladi, \emph{Partitions with non-repeating odd parts and combinatorial identities}, Ann. Comb. {\bf 20} (2016), 1--20.

\bibitem{AllBer}
K. Alladi, A. Berkovich, \emph{New polynomial analogues of Jacobi's triple product and Lebesgue's identities}, Adv. Appl. Math. {\bf 32} (2004), 801--824.

\bibitem{Abook}
G. E. Andrews, \emph{The Theory of Partitions}, Addison--Wesley, Reading, MA, 1976; reissued: Cambridge University Press,
Cambridge, 1998.

\bibitem{And_aGauss}
G. E. Andrews, \emph{$a$-Gaussian polynomials and finite Rogers-Ramanujan identities}, In Theory and Applications of Special Functions: A Volume Dedicated to Mizan Rahman, M. Ismail and E. Koelink eds., 39--60. Springer, New York, 2005.

\bibitem{Delannoy}
C. Banderier, S. Schwer, \emph{Why Delannoy numbers?},
J. Statist. Plann. Inference {\bf 135} (2005), no. 1, 40--54.

\bibitem{Butler}
L. M. Butler, \emph{The $q$-log-concavity of $q$-binomial coefficients},
J. Combin. Theory Ser. A {\bf 54} (1990), no. 1, 54--63.

\bibitem{LC}
S. Corteel, J. Lovejoy, \emph{Overpartitions}, Trans. Amer. Math. Soc. \textbf{356} (2004), no. 4, 1623--1635.

\bibitem{DK}
J. Dousse and B. Kim, \emph{An overpartition analogue of q-binomial coefficients}, Ramanujan J. {\bf 42} (2017),  267--283.

\bibitem{Fine}
N. J. Fine, \emph{Basic Hypergeometric Series and Applications}, American Mathematical Society, Providence, RI, 1988.

\bibitem{GR}
G. Gasper and M. Rahman, Basic Hypergeometric Series, 2nd Edition, Cambridge Univ. Press, Cambridge, 2004.

\bibitem{Pak}
I. Pak, G. Panova, \emph{Strict unimodality of q-binomial coefficients},
Comptes Rendus Math\'ematiques {\bf 351} (2013), 415--418. 

\bibitem{PS}
T. Prellberg and D. Stanton, \emph{Proof of a monotonicity conjecture},
J. Combin. Theory Ser. A {\bf 103} (2003), no. 2, 377--381. 

\bibitem{Ramirez}
J. L. Ramirez, \emph{Incomplete Tribonacci numbers and polynomials}, J. Integer Seq. {\bf 17} (2014), article 14.4.2.

\bibitem{Rowell}
M. Rowell, \emph{A new exploration of the Lebesgue identity}, Int. J. Number Theory {\bf 6} (2010),  785--798.

\bibitem{Sagan}
B. Sagan, \emph{Unimodality and the reflection principle}.  
Ars Combin. {\bf 48} (1998), 65--72. 


\bibitem{Syl}
J. J. Sylvester, \emph{A constructive theory of partitions, arranged in three acts, an interact, and an exodion}, in The Collected Papers of J. J. Sylvester, Vol. 3, Cambridge University Press, London, 1--83; reprinted by Chelsea, New York, 1973.

\bibitem{Syl2}
J. J. Sylvester, \emph{Proof of the hitherto undemonstrated
fundamental theorem of invariants}, Philosophical Magazine {\bf 5} (1878), 178--188.

\bibitem{Za} 
D. Zagier, \emph{Vassiliev invariants and a strange identity related to the Dedekind eta-function}, Topology {\bf 40} (2001), 945--960.

\end{thebibliography}
\end{document}